\documentclass[12pt, letterpaper]{amsart}

\usepackage{amsmath,amssymb,latexsym, color
}
\usepackage[dvips,centering,includehead,width=13.1cm,height=22.7cm,
paperheight=23.8cm,paperwidth=17cm
]{geometry}
\usepackage{enumitem}
\usepackage{hyperref}
\usepackage{pict2e}
\usepackage{epic,graphicx}
\usepackage{rotating}
\usepackage{mathtools}
\usepackage{tikz-cd}
\usepackage{ccaption}
\usepackage[commentmarkup=footnote]{changes}
\usepackage{subfiles}
\usepackage{nicefrac}
\changecaptionwidth
\captionwidth{5.6in}

\renewcommand{\ref}{\hyperref}

\def\K{{\mathbb K}}
\def\Z{{\mathbb Z}}
\def\R{{\mathbb R}}
\def\Q{{\mathbb Q}}

\def\C{{\mathbb C}}

\def\PP{{\mathbb P}}

\newcommand{\bi}{\boldsymbol{i}}

\newcommand{\proofend}{\hfill$\Box$\bigskip}

\DeclareMathOperator{\pr}{pr}\DeclareMathOperator{\St}{St}
\DeclareMathOperator{\val}{Val}

\DeclareMathOperator{\Sing}{\mathbf{Sing}}

\newtheorem{theorem}{Theorem}[section]
\newtheorem{proposition}[theorem]{Proposition}

\newtheorem{lemma}[theorem]{Lemma}

\theoremstyle{definition}

\newtheorem{remark}[theorem]{Remark}

\newtheorem{definition}[theorem]{Definition}

\numberwithin{equation}{section}

\newcommand{\Id}{{ \operatorname{Id}}}

\usetikzlibrary{arrows.meta, decorations.markings}

\usepackage{environ}
\makeatletter
\newsavebox{\measure@tikzpicture}
\NewEnviron{scaletikzpicturetowidth}[1]{%
	\def\tikz@width{#1}%
	\begin{lrbox}{\measure@tikzpicture}%
		\BODY
	\end{lrbox}%
	\pgfmathparse{#1/\wd\measure@tikzpicture}%
	\BODY
}
\makeatother

\usepackage{todonotes}

\newcommand{\todos}{\makeatletter
	\providecommand\@dotsep{5}
	\makeatother
	\listoftodos\relax}

\DeclarePairedDelimiterX\setc[2]{\{}{\}}{\,#1 \;\delimsize\vert\; #2\,}

\begin{document}
\title{Anti-Zariski pairs}
\author{Peng Ren}
\address{School of Mathematical Sciences, Tel Aviv
	University, Tel Aviv 6997801,
	Israel}
\email{pren@fudan.edu.cn}

\author{Eugenii Shustin}
\address{School of Mathematical Sciences, Tel Aviv
	University, Tel Aviv 6997801,
	Israel}
\email{shustin@tauex.tau.ac.il}


\thanks{The authors were supported by the DFG grant no. DE 1422/9-1 and by the BSF grant no. 2022/157.
The first author was partly supported by a postdoctoral fellowship from the Shamoon College of Engineering. The second author enjoyed a support from the Bauer-Neuman Chair in Real and Complex Geometry.}


\begin{abstract}
In 1929, O. Zariski found a pair of complex plane algebraic curves of the same degree and with the same collection of singularities, but embedded into the plane in a topologically different way. Accordingly, such curves belong to different components of the equisingular family. This phenomenon has been intensively studied till now. In this note, we propose a different insight on this subject: Two curves $C',C''\subset\PP^2$ form an {\it anti-Zariski pair}, if $(\PP^2,C')$ and $(\PP^2,C'')$ are homeomorhic, but $C'$ and $C''$ belong to different components of the equisingular family. We exhibit examples of anti-Zariski pairs and discuss related issues.	
\end{abstract}


\maketitle

\tableofcontents

\makeatletter
\providecommand\@dotsep{5}
\makeatother
	
\section{Introduction}

Starting with the famous Zariski example of two complex plane sextic curves with six cusps, whose complements have different fundamental groups \cite{Za}, the study of this phenomenon has become an intensive research area including computation of various topological and algebraic-geometric invariants of plane algebraic curves, related singularity theory, topology of algebraic surfaces, and other issues.

The goal of this note is to attract the attention of experts to a somewhat opposite phenomenon, which we call {\it anti-Zariski pairs},
i.e., pairs of complex plane curves of the same degree which are topologically equivalent as embedded curves, but are not isotopic in a family of algebraic curves. Here is the precise definition.

\begin{definition}\label{def-azp}
Two reduced curves $C,C'\subset\PP^2$ of the same degree form an anti-Zariski pair if they belong to different components of the same equisingular family and the pairs $(\PP^2,C)$, $(\PP^2,C')$ are homeomorphic.
\end{definition}

\begin{remark}\label{rem-weak}
In fact, anti-Zariski pairs have been known for a while, for instance, pairs of branch curves of generic projections of certain algebraic surfaces \cite{Cat,KK} and pairs of maximizing sextic curves \cite{AD,ABCC,Deg12}.
\end{remark}

One can consider pairs of curves satisfying relaxed requirements:
We say that curves $C,C'\subset\PP^2$ form a weak anti-Zariski pair, if they belong to different components of the same equisingular family and the complements $\PP^2\setminus C$, $\PP^2\setminus C'$ are homeomorphic.

Note that there are many pairs of curves $C,C'\subset\PP^2$ of the same degree, belonging to different components of the equisingular families and having homotopy equivalent complements $\PP^2\setminus C$, $\PP^2\setminus C'$: By \cite[Proposition 6.5]{Deg14} the complement to an irreducible plane curve $D\subset\PP^2$ of degree $m$ such that $\pi_1(\PP^2\setminus D)\simeq\Z_m$ has the standard homotopy type depending on $m$ and on the total Milnor number $\mu(D)$ of all singular points. Examples of such pairs of sextic curves with $\mu=19$ can be found in \cite{Deg12,Deg14}, and some of them are anti-Zariski pairs (see Section \ref{sec-sextic} for details). Similar examples of pairs of curves with nodes and cusps and of curves with ordinary multiple points whose complements have abelian fundamental groups can be found in \cite{GLS2000} and \cite{GLS01}, respectively, but we do not know whether they form anti-Zariski pairs (even in a weak sense). On the other hand, in \cite{ACM} and \cite{Shir}, one can find infinite series of tuples of reducible equisingular curves of the same degree whose complements have fundamental groups isomorphic to $\Z^3$, but are not homeomorphic to each other. Examples of pairs of line arrangements which are not equivalent up to a homeomorphism of the plane, but have homotopy equivalent complements are found in \cite[Theorem 1.4]{Gue}.

All anti-Zariski pairs presented in this note are pairs of complex conjugate curves. A natural question arises:

\smallskip

\smallskip{\bf Question:} {\it Is there an anti-Zariski pair of irreducible curves $C,C'\subset\PP^2$ such that $C'$ and the complex conjugate of $C$ belong to different components of the equisingular family?}

\medskip

Possible candidates for such anti-Zariski pairs are the aforementioned examples in \cite{GLS2000,GLS01} though they seem hard to handle due to relatively high degree. Much more reasonable candidates are pairs of equisingular maximizing sextic curves that are not complex conjugate and whose complements have isomorphic fundamental groups, see \cite{AD}, \cite{ABCC},  \cite[Chapter 8]{Deg12}, and \cite{Deg14}. Other reasonable candidates are pairs of line arrangements of $13$ lines constructed in \cite[Section 3]{GB} and irreducible curves of any even degree $\ge6$ having three higher cusps \cite[Proposition 1.1 and Question 4.1]{ACM}.

\smallskip

The paper is organized as follows. In Sections \ref{sec-bc} and \ref{sec-sextic}, we exhibit examples of anti-Zariski pairs found in the existing literature. In Sections \ref{sec-la} and \ref{sec-cla}, we present new examples: anti-Zariski pairs formed by line arrangements of degree $10$ and conic-line arrangements of degree $9$ and $10$, and we show that $10$ is the minimal degree of line arrangements in an anti-Zariski pair. In Section \ref{sec-red}, we show a simple way to produce infinitely many anti-Zariski pairs starting with any given anti-Zariski pair. Section \ref{sec-na} is devoted to a version of the problem stated over an algebraically closed, non-Archimedean extension of the complex field $\C$. We also pose some open questions.

\smallskip{\bf Acknowledgements.} We are grateful to Meirav Amram, Michael Dettweiler, and Junhan Tang for stimulating discussions on the subject. Special thanks are due to Enrique Artal Bartolo and Alex Degtyarev for their very valuable comments on the preliminary version of this paper. The first author also thanks Meirav Amram and Michael Dettweiler for the guidance and support in his postdoctoral research; he also would like to thank the University of Bayreuth for hospitality and excellent working conditions during his visit there.

\section{Preliminaries}

\subsection{Equisingularity relation}
We work over the complex field $\C$.

Let $\PP^2$ be the complex projective plane with the standard complex conjugation $c:\PP^2\to\PP^2$. For any geometric object $T\subset\PP^2$, we denote its complex conjugate by $T^c$.

\begin{definition}\label{def-top}
Two plane curve singular points $(C,p)$ and $(C',p')$ are called topologically equivalent if there are small open balls $p\in B\subset\C^2$ and $p'\in B'\subset \C^2$ and a homeomorphism $h:B\to B'$ such that
$$h(p)=p'\quad\text{and}\quad h(B\cap C)=B'\cap C'.$$
\end{definition}

\begin{definition}\label{def-es}
(1) Two reduced, irreducible curves $C,C'\subset\PP^2$ of the same degree are called equisingular (to each other) if there is a bijection $\quad$ $\xi:\Sing(C)\to\Sing(C')$ such that the singular points $(C,p)$ and $(C',\xi(p))$ are topologically equivalent for all $p\in\Sing(C)$.

(2) Two reduced, reducible curves $C,C'\subset\PP^2$ of the same degree are called equisingular if there are two bijections $\xi:\Sing(C)\to\Sing(C')$ and $\eta:\{C_1,...,C_r\}\to\{C'_1,...C'_r\}$, $\eta(C_i)=C'_i$, $1\le i\le r$, where $C_1,...,C_r$ are all (distinct) components of $C$, and $C'_1,...,C'_r$ are all (distinct) components of $C'$, and the following holds:
\begin{itemize}\item $\deg C_i=\deg C'_i$, $1\le i\le r$;
\item $p\in\Sing(C)$ is a common point of $C_{i_1},...,C_{i_k}$ if and only if $\xi(p)\in\Sing(C')$ is a common point of $C'_{i_1},...,C'_{i_k}$, and moreover, for any non-empty subset $K\subset\{1,...,k\}$, the germs
    $$\left(\bigcup_{j\in K}C_{i_j},p\right)\quad\text{and}\quad\left(\bigcup_{j\in K}C'_{i_j},\xi(p)\right)$$
    are topologically equivalent.
\end{itemize}

(3) An equivalence class of the equisingularity relation is called an {\it equisingular family} and is denoted $ES(C)$, $C$ being any representative.
\end{definition}

It is well-known that equisingular families $ES(C)$ are locally closed subschemes of the linear systems $|{\mathcal O}_{\PP^2}(d)|$, $d=\deg C$ (see, for example \cite[Section 2.2.2.1]{GLS}).

A curve $C\subset\PP^2$ is called rigid if the component of $ES(C)$ containing $C$ is the orbit of the $PGL(3,\C)$-action.

\subsection{Example: Equisingular families of line arrangements}
Traditionally, line arrangements are assumed to be ordered sequences of distinct lines in the plane, and they are parameterized by the realization spaces $R(I)$, where $I$ is the incidence relation (see details in \cite{NY,Ye}).

The order of lines does not matter for the equisingular relation, and we replace a line arrangement ${\mathcal A}$ with $|{\mathcal A}|$, the union of lines of ${\mathcal A}$. To relax notations, we write $R({\mathcal A})$ and $ES({\mathcal A})$ instead of $R(I({\mathcal A}))$ and $ES(|{\mathcal A}|)$, respectively. These spaces are related by the formula
\begin{equation}ES({\mathcal A})=\pr(R({\mathcal A})),\label{e-rs-esf}\end{equation}
where $\pr:((\PP^2)^*)^n\to|{\mathcal O}_{\PP^2}(n)|$ is given by
$$(l_1,...,l_n)\in((\PP^2)^*)^n\ \overset{\pr}{\mapsto}\ l_1\cup...\cup l_n\in|{\mathcal O}_{\PP^2}(n)|.$$

We also mention a simple criterion for two line arrangements to be equisingular:

\begin{lemma}\label{la-es}
Two (unordered) line arrangements are equisingular if and only if they are homeomorphic.
\end{lemma}

\section{Examples of anti-Zariski pairs}

\subsection{Anti-Zariski pairs of branch curves}\label{sec-bc}
The following statement is an immediate consequence of \cite[Proposition 4.3]{KK}.

\begin{proposition}\label{prop-surf1}
Let $X,X^c\subset\PP^r$ be smooth projective complex conjugate surfaces of general type that are not deformation equivalent, $\pr:X\to\PP^2$ and $\pr^c:X^c\to\PP^2$ two generic complex conjugate projections. Then the branch curves $B$ and $B^c$ of these projections form an anti-Zariski pair.
\end{proposition}

Examples of such pairs of surfaces come from \cite[Theorem 1.3]{Cat} and \cite[Theorem 1.1 and Proposition 2.1]{KK}.

\subsection{Anti-Zariski pairs of irreducible sextic curves}\label{sec-sextic}
Equisingular families of reduced curves of degree $\le5$ are irreducible, see \cite{Deg90} and \cite[Theorem 7.49]{Deg12}; hence, anti-Zariski pairs of reduced curves should be of degree at least $6$. Here, we present examples of anti-Zariski pairs of the minimal possible degree $6$.

U.Persson \cite{Per} showed that
\begin{itemize}\item The maximal possible total Milnor number of a plane irreducible sextic curve with simple ADE singularities equals $19$ (such sextics are called {\it maximizing}).
\item Maximizing sextics are rigid (this fact can also be derived from \cite[Theorem 1]{Sh}).
Furthermore, each class is represented by a curve defined over a number field $\Q(\lambda)$ with some algebraic number $\lambda$.
\end{itemize}
The moduli spaces of maximizing sextics were classified by I. Shimada \cite{Shi}. Explicit equations of the representatives with only double points were obtained by S. Orevkov \cite{Ore}. Among them there are pairs of complex conjugate representatives, which we list in the following statement (see \cite[Tables 2 and 3]{AD}):

\begin{proposition}\label{prop-sextic}
There are anti-Zariski pairs of complex conjugate maximizing sextic curves possessing the following collections of singularities:
\begin{footnotesize}
\begin{align} &A_{18}\oplus A_1 &-\ \text{1 pair};\nonumber\\
&A_{16}\oplus A_2\oplus A_1 &-\ \text{1 pair};\nonumber\\
&A_{15}\oplus A_4 &-\ \text{1 pair};\nonumber\\
&A_{14}\oplus A_4\oplus A_1 &-\ \text{3 pairs};\nonumber\\
&A_{13}\oplus A_6 &-\ \text{2 pairs};\nonumber\\
&A_{12}\oplus A_7 &-\ \text{1 pair};\nonumber\\
&A_{12}\oplus A_6\oplus A_1 &-\ \text{1 pair};\nonumber\\
&A_{12}\oplus A_4\oplus A_2\oplus A_1 &-\ \text{1 pair};\nonumber\\
&A_{10}\oplus A_8\oplus A_1 &-\ \text{1 pair};\nonumber\\
&A_{10}\oplus A_6\oplus A_3 &-\ \text{1 pair};\nonumber\\
&A_{10}\oplus A_6\oplus A_2\oplus A_1 &-\ \text{1 pair};\nonumber\\
&A_{10}\oplus 2A_4\oplus A_1 &-\ \text{1 pair};\nonumber\\
&A_9\oplus A_6\oplus A_4 &-\ \text{1 pair};\nonumber\\
&A_8\oplus A_7\oplus A_4 &-\ \text{1 pair};\nonumber\\
&A_8\oplus A_6\oplus A_4\oplus A_1 &-\ \text{1 pair};\nonumber\\
&A_8\oplus A_5\oplus A_4\oplus A_2 &-\ \text{1 pair};\nonumber\\
&A_7\oplus 2A_6 &-\ \text{1 pair}.\nonumber
\end{align}
\end{footnotesize}
\end{proposition}

\begin{remark}\label{rem-sec}
(1) Similar examples of anti-Zariski pairs of maximizing sextics with a triple point can be extracted from \cite[Tables 1 and 4]{AD}. In addition, examples of anti-Zariski pairs of reducible, two-component sextics can be extracted from \cite[Theorems 1 and 2]{ABCC} and \cite[Table 8.11]{Deg12}.

(2) According to \cite{AD,Deg12,Deg14}, the fundamental groups of the complements to real maximizing sextics with singularity collections as in Proposition \ref{prop-sextic} are isomorphic to $\Z_6$. The fundamental groups of the complements to other maximizing sextics can be non-abelian, in particular, this holds for the anti-Zariski pair of sextics with singularities $A_8\oplus A_5\oplus A_4\oplus A_2$ (see more details in \cite{AD}). It would be interesting to compute the fundamental groups of the complements to sextic curves in other anti-Zariski pairs listed in Proposition \ref{prop-sextic}.
\end{remark}

\subsection{Anti-Zariski pairs of line arrangements}\label{sec-la}
Here, we present an anti-Zariski pair of line arrangements of degree $10$ and show that $10$ is the minimal degree of line arrangements in anti-Zariski pairs.

We start with the self-dual MacLane arrangement (aka M\"obius-Kantor arrangement) ${\mathcal L}\subset\PP^2$ of type $(8_3)$ (see \cite{CoX,Dol} and \cite[Example 4.3]{NY}): It consists of $8$ points and $8$ lines such that each point is common to $3$ lines and each line contains $3$ points (the four extra double points are not counted). It is well-known that ${\mathcal L}$ is rigid. We choose the following representative ${\mathcal L}_0\in ES({\mathcal L})$ given in the projective coordinates $x,y,z$ as follows:
\begin{itemize}\item real points
\begin{align}&p_1=(0,1,0),\ p_2=(-1,1,1),\ p_3=(-1,-1,1),\nonumber\\ &p_4=(1,1,1),\ p_5=(1,-1,1),
\nonumber\end{align}
non-real points
$$p_6=(\bi\sqrt{3},1,1),\ p_7=(\bi\sqrt{3},-1,1),\ p_8=(\bi,-\bi,\sqrt{3});$$
\item real lines
\begin{align}&l_1=\{x+z=0\},\ l_2=\{x-z=0\},\ l_3=\{y+z=0\},\nonumber\\ &l_4=\{y-z=0\},\ l_5=\{x+y=0\},\nonumber\end{align}
non-real lines
\begin{align}&l_6=\{x-\bi\sqrt{3}z=0\},\quad l_7=\{\omega^2x-\omega y+z=0\},\nonumber\\ &l_8=\{(\omega^c)^2x-\omega^c y-z=0\},\quad \omega=\frac{1+\bi\sqrt{3}}{2};\nonumber\end{align}
\item the incidence relations
\begin{align}&p_1,p_2,p_3\in l_1,\ p_1,p_4,p_5\in l_2,\ p_3,p_5,p_7\in l_3,\ p_2,p_4,p_6\in l_4,\nonumber\\
&p_2,p_5,p_8\in l_5,\ p_1,p_6,p_7\in l_6,\ p_4,p_7,p_8\in l_7,\ p_3,p_6,p_8\in l_8.\nonumber\end{align}
\end{itemize}

\begin{proposition}\label{prop-la}
The pair $|C_{10}|,|C^c_{10}|$ is an anti-Zariski pair, where $C_{10}$ is the union of ${\mathcal L}_0$ and two additional lines $l_9$ and $l_{10}$ such that $l_9$ passes through $p_6$ and a point $p\in l_1\setminus\{p_1,p_2,p_3\}$, while $l_{10}$ passes through $p_7$ and $p$.
\end{proposition}

{\bf Proof.}
Suppose on the contrary that $|C_{10}|$ and $|C^c_{10}|$ are joined by an equisingular deformation.
Let us label the lines $l_1,...,l_{10}$ by sequences of multiplicities of points lying on them (while the double points are ignored)
$$l_1:(3,3,3,3),\ l_2:(3,3,3),\ l_3:(3,3,4),\ l_4:(3,3,4),\ l_5:(3,3,3),$$
$$l_6:(3,4,4),\ l_7:(3,3,4),\ l_8:(3,3,4),\ l_9:(3,4),\ l_{10}:(3,4).$$
Observe that the labels $(3,4)$ of $l_9$ and $l_{10}$ differ from the labels of $l_1,...,l_8$; hence, the above equisingular deformation descends to an equisingular deformation of $|{\mathcal L}_0|$ to $|{\mathcal L}_0^c|$.
The rigidity of $|{\mathcal L}_0|$ implies that the latter equisingular deformation is induced by a deformation inside $PGL(3,\C)$ of $\Id$ to $\varphi$ taking $|{\mathcal L}_0|$ to $|{\mathcal L}_0^c|$. Furthermore, the lines $l_1$ and $l_6$ have unique labels $(3,3,3,3)$ and $(3,4,4)$, respectively, and hence
$$\varphi(l_1)=l_1,\quad\varphi(l_6)=l_6^c,\quad \varphi(p_1)=p_1.$$
Since $l_2$ is the only remaining line through $p_1$, we derive that $\varphi(l_2)=l_2$. Among the other lines $l_3,l_4,l_5,l_7,l_8$, the line $l_5$ has its unique label $(3,3,3)$, and then $\varphi(l_5)=l_5$, which subsequently yields that $\varphi(p_2)=p_2$ and $\varphi(p_5)=p_5$. Next, we conclude
\begin{align}&\varphi(l_1)=l_1,\ \varphi(p_1)=p_1,\ \varphi(p_2)=p_2\ &\Longrightarrow\quad&\varphi(p_3)\in\{p_3,p^c\},\nonumber\\
&\varphi(l_2)=l_2,\ \varphi(p_1)=p_1,\ \varphi(p_5)=p_5\ &\Longrightarrow\quad&\varphi(p_4)=p_4.\nonumber\end{align}
At last, we have $\varphi(p_3)\ne p^c$, since otherwise, the line $l_3$ will go to the line through $p_5$ and $p^c$, which is not a part of ${\mathcal L}_0^c$.

Thus, the four points $p_2,p_3,p_4,p_5$ in general position are fixed by $\varphi$, and hence $\varphi=\Id$, which contradicts the initial assumption.
\proofend

\begin{remark}
One can construct similar examples of anti-Zariski pairs starting with any rigid line arrangement necessarily containing non-real points and lines.
\end{remark}

Now, we show that the example of an anti-Zariski pair of line arrangements in Proposition \ref{prop-la} has the minimal possible degree, i.e., there are no anti-Zariski pairs formed by line arrangements of degree $\le9$.

\begin{proposition}\label{prop-9}
Equisingular families of line arrangements of degree $\le9$ are irreducible.
\end{proposition}

{\bf Proof.}
We explore results of \cite{NY,Ye} on irreducibility of realization spaces of line arrangements.
By formula (\ref{e-rs-esf}), for any line arrangement ${\mathcal A}$, the irreducibility of $R({\mathcal A})$ implies the irreducibility of $ES({\mathcal A})$.

By \cite[Propositions 3.3, 4.5, and 4.6, and Theorem 3.15]{NY} the realization spaces of all line arrangements of degree $\le8$ are irreducible, except for the MacLane arrangement ${\mathcal L}_0$. However, it is known that $ES({\mathcal L}_0)$ is the orbit of the $PGL(3,\C)$-action, and hence irreducible.

Furthermore, by \cite[Theorem 3.9]{Ye}, any arrangement of $9$ lines satisfies the following:
\begin{itemize}\item either its realization space is irreducible, and hence its equisingular family is irreducible;
\item or it is a Falk-Sturmfels arrangement (see \cite[Example 5.2]{NY} and \cite[Example 2.3]{Ye}); according to \cite[Example 5.2]{NY}, the realization space of Falk-Sturmfels arrangements consists of two components represented by arrangements ${\mathcal{FS}}^+$ and ${\mathcal{FS}}^-$, and there exists a projective transformation of the plane taking $|{\mathcal{FS}}^+|$ to $|{\mathcal{FS}}^-|$; hence, the equisingular family of Falk-Sturmfels arrangements is irreducible;
\item or it is an ${\mathcal A}^{\pm\bi}$ arrangement (see \cite[Example 5.3]{NY} and \cite[Example 2.4]{Ye}), where $|{\mathcal A}^{\bi}|$ is given by (cf. \cite[Example 2.4]{Ye})
 \begin{align}xy(x-z)(y-z)&(x-\bi z)(y-\bi z)(x-y)\nonumber\\
 &\times((\bi-1)x+\bi y+z)((1-\bi)x+y-z)=0,\nonumber\end{align}
 and ${\mathcal A}^{-\bi}=({\mathcal A}^{\bi})^c$; here, ${\mathcal A}^{\bi}$ and ${\mathcal A}^{-\bi}$ represent the only two components of the realization space; however, it is easy to see that the projective transformation
 $(x,y,z)\mapsto(x,y,\bi z)$ takes ${\mathcal A}^{\bi}$ to ${\mathcal A}^{-\bi}$, and hence the equisingular family is irreducible;
\item or it contains a MacLane arrangement.
\end{itemize}
It remains to consider the latter case, namely, arrangements ${\mathcal A}_9=({\mathcal L}_0,l_9)$ with a certain line $l_9\not\subset|{\mathcal L}_0|$.
Suppose that $l_9=\{y+\bi\sqrt{3}z=0\}$. Then it passes through all four double points
$$q_0=(1,0,0),\ q_1=(-1,-\bi\sqrt{3},1),\ q_2=(1,-\bi\sqrt{3},1),\ q_3=(\bi\sqrt{3},-\bi\sqrt{3},1)$$
of ${\mathcal L}_0$, and ${\mathcal A}_9$ turns out to be the dual Hesse arrangement ${\mathcal H}^\vee$ (i.e., the arrangement dual to the configuration of inflection points of a smooth plane cubic curve), which is rigid and $ES({\mathcal H}^\vee)$ is an irreducible orbit of the $PGL(3,\C)$-action.

Suppose that $l_9$ passes through a double point $q$ and a triple point $p$ of ${\mathcal L}_0$. For a given double point $q$, there are two suitable triple points $p$. Thus, to prove the irreducibility of $ES({\mathcal A}_9)$, we have to show that the stabiliser $\St(|{\mathcal L}_0|)\subset PGL(3,\C)$ acts transitively on the set of eight appropriate pairs $(p,q)$. Consider the transformation $\varphi\in PGL(3,\C)$ such that
$$\varphi(p_2)=p_4,\quad\varphi(p_4)=p_6,\quad\varphi(p_3)=p_5,\quad\varphi(p_5)=p_7.$$
Then $\varphi(l_3)=l_3$ and $\varphi(l_4)=l_4$, which also implies that $\varphi(q_0)=q_0$ since $\{q_0\}=l_3\cap l_4$. Since cross-ratios of four-tuples of points on lines are preserved, and since
$$\begin{cases}cr(q_0,p_2,p_4,p_6)=\frac{1+\bi\sqrt{3}}{2}=cr(q_0,p_4,p_6,p_2),&\\
cr(q_0,p_3,p_5,p_7)=\frac{1+\bi\sqrt{3}}{2}=cr(q_0,p_5,p_7,p_3)&,\end{cases}$$
we conclude that $\varphi(p_7)=p_3$ and $\varphi(p_6)=p_2$. It follows that
$$\varphi(l_5)=l_7,\ \varphi(l_7)=l_8,\ \varphi(l_8)=l_5,\varphi(l_1)=l_2,\ \varphi(l_2)=l_6,\ \varphi(l_6)=l_1;$$
hence, $\varphi\in\St(|{\mathcal L}_0|)$, and it induces the permutation
$$(q_0,q_1,q_2,q_3)\mapsto(q_0,q_2,q_3,q_1).$$
Selecting another double point $q_i$, $1\le i\le 3$, and exploring the two lines of ${\mathcal L}_0$ through it as $l_3,l_4$ above, we show that $\St(|{\mathcal L}_0|)$ induces all even permutations of the four-tuple of the double points, and acts transitively on it. Now, let us choose the double point $q_0$. The two suitable triple points for it are $p_1,p_8$. We will exhibit a transformation $\psi\in \St(|{\mathcal L}_0|)$ that fixes $q_0$ and interchanges $p_1$ and $p_8$. Namely, define $\psi\in PGL(3,\C)$ by
$$\psi(p_2)=p_7,\quad\psi(p_7)=p_2,\quad\psi(p_5)=p_6,\quad\psi(p_6)=p_5.$$
The above cross-ratio argument ensures that $\psi(p_3)=p_4$ and $\psi(p_4)=p_3$. It follows that
$$l_3\overset{\psi}{\leftrightarrow}l_4,\quad l_5\overset{\psi}{\leftrightarrow}l_6,\quad
l_7\overset{\psi}{\leftrightarrow}l_1,\quad l_8\overset{\psi}{\leftrightarrow}l_2$$
Thus, $\psi\in\St(|{\mathcal L}_0|)$ and $\psi(q_0)=q_0$, $\psi(p_1)=p_8$, $\psi(p_8)=p_1$ as required.

Suppose that $l_9$ passes through a triple or a double point of ${\mathcal L}_0$ and no other intersection point of lines of ${\mathcal L}_0$, then $ES({\mathcal A}_9)$ can be regarded as a fibration over
$${\mathcal F}_3=\{(p,{\mathcal L})\in\PP^2\times ES({\mathcal L}_0)\ :\ p\ \text{a triple point of}\ {\mathcal L}\},$$
or
$${\mathcal F}_2=\{(q,{\mathcal L})\in\PP^2\times ES({\mathcal L}_0)\ :\ q\ \text{a double point of}\ {\mathcal L}\},$$ respectively,
whose fibers are Zariski open subsets of pencils of lines. Since $\St(|{\mathcal L}_0|)$ acts transitively on the set of triple points and on the set of double points of ${\mathcal L}_0$ (as we explained in the preceding paragraph), the coverings
$${\mathcal F}_3\to ES({\mathcal L}_0), \quad {\mathcal F}_2\to ES({\mathcal L}_0)$$ are irreducible; hence, $ES({\mathcal A}_9)$ is irreducible.

At last, suppose that $l_9$ avoids all intersection points of lines of ${\mathcal L}_0$. Then $ES({\mathcal A}_9)$ can be viewed as a fibration over $ES({\mathcal L}_0)$, whose fibers are Zariski open subsets of $(\PP^2)^*$. Thus, $ES({\mathcal A}_9)$ is irreducible.
\proofend

\subsection{Anti-Zariski pairs of conic-line arrangements}\label{sec-cla}
Here, we construct anti-Zariski pairs of conic-line arrangements of degree $9$ and $10$.

\begin{proposition}\label{prop-cla}
The pair $|D_{10}|,|D^c_{10}|$ is an anti-Zariski pair, where $D_{10}$ is the union of ${\mathcal L}_0$ and a generic conic $C_2$ tangent to the lines $l_6,l_7,l_8$ and passing through the point $p_2$.
\end{proposition}

{\bf Proof.}
Suppose on the contrary that $|D_{10}|$ and $|D^c_{10}|$ are joined by an equisingular deformation.
Clearly, this equisingular deformation descends to an equisingular deformation of $|{\mathcal L}_0|$ to $|{\mathcal L}_0^c|$, and hence to a deformation inside $PGL(3,\C)$ of $\Id$ to $\varphi$ taking $|{\mathcal L}_0|$ to $|{\mathcal L}^c_0|$. Since the lines $l_6,l_7,l_8$ are the only lines in $D_{10}$ tangent to the conic $C_2$, $\varphi$ takes the union of real lines $l_1\cup...\cup l_5$ to itself. The latter union contains a unique four-tuple of lines in general position, and hence $\varphi$ is represented by a real $3\times3$ matrix. It implies that $\varphi(p_2)=p_2$ as $p_2$ is the only four-fold real point in $D_{10}$, and $\varphi(l_1\cup l_2)=l_1\cup l_2$, since $l_1,l_2$ are the only real lines in ${\mathcal L}_0$ passing through three real points of multiplicity $\ge3$.
It follows that $\varphi(p_1)=p_1$, $\varphi(l_1)=l_1$, and $\varphi(p_3)=p_3$. Furthermore, $\varphi(l_2)=l_2$, and $\varphi$ cannot interchange the points $p_4,p_5$, since otherwise the line $l_3$ would go to the line through $p_3$ and $p_4$ which is not a part of ${\mathcal L}^c_0$. Thus, $\varphi(p_4)=p_4$, $\varphi(p_5)=p_5$, and we conclude that $\varphi=\Id$ since it fixes four points in general position $p_2,p_3,p_4,p_5$.
\proofend

For the next example, we recall that the standard quadratic Cremona transformation of the plane is a birational automorphism of $\PP^2$ defined by
$$x'=f_2f_3,\quad y'=f_1f_3,\quad z'=f_1f_2,$$
where $f_1,f_2,f_3\in\C[x,y,z]$ are linearly independent linear forms. The lines defined by these forms are called the fundamental lines and their intersection points are called the fundamental points of the transformation. Geometrically, it is composed of the blowing-up of the fundamental points and then the contraction of the strict transforms of the fundamental lines. The images of the contracted lines serve as the fundamental points of the inverse quadratic Cremona transformation. A curve of degree $m$ not containing the fundamental lines as components is transformed to a curve of degree $2m-m_1-m_2-m_3$, where $m_1,m_2,m_3$ are the multiplicities of the fundamental points in the given curve. We denote the images of points and curves under the quadratic Cremona transformation by the upper asterisk.

Now, we take the line arrangement $C_{10}$ from Proposition \ref{prop-la}, reduce it to the subarrangement $C_7$ including the lines $l_1,l_2,l_3,l_4,l_5,l_7,l_8$, and apply the quadratic Cremona transformation with the fundamental lines $l_6,l_9,l_{10}$ (correspondingly, $p,p_6,p_7$ are the fundamental points). Thus, the lines $l_1,l_3,l_4,l_7$, and $l_8$ are taken to lines $l_1^*,l_3^*,l_4^*,l_7^*$, and $l_8^*$, respectively, and the lines $l_2$ and $l_5$ are taken to conics $l_2^*$ and $l_5^*$, respectively. This conic-line arrangement has degree $9$ and the following points of multiplicity $\ge3$:
\begin{itemize}\item the ordinary triple points $p_2^*$, $p_3^*$, $p_4^*$, $p_5^*$, and $p_8^*$ coming from the triple points of $C_7$ outside $l_6\cup l_9\cup l_{10}$;
\item the points $l_9^*$ (intersection of $l_2^*,l_3^*,l_5^*,l_7^*$) and $l_{10}^*$ (intersection of $l_2^*,l_4^*,l_5^*,l_8^*$) of multiplicity $4$ and the point $l_6^*$, where the conics $l_2^*$ and $l_5^*$ intersect transversally, and the line $l_1^*$ is tangent to $l_2^*$ (type $D_6$ in the classification of plane curve singularities);
    \item three transverse double intersection points.
\end{itemize}

\begin{proposition}\label{prop-cla1}
The pair $|D_9|,|D_9^c|$ is an anti-Zariski pair.
\end{proposition}

{\bf Proof.}
If we apply to $|D_9|$ the inverse quadratic Cremona transformation with the fundamental points $l_6^*,l_9^*,l_{10}^*$, we obtain the line arrangement $|C_7|$ which gives $|C_{10}|$ after adding the original fundamental lines. Correspondingly, the image of $|D_9^c|$ under the quadratic Cremona transformation with the fundamental points $(l_6^c)^*,(l_9^c)^*,(l_{10}^c)^*$ is $|C_7^c|$ which similarly extends to $|C_{10}^c|$.
Observe that the two four-fold points and the point of type $D_6$ are uniquely determined for any conic-line arrangement equisingular to $|D_9|$. Hence, an equisingular deformation of $|D_9|$, resp. $|D_9^c|$, yields an equisingular deformation of $|C_{10}|$, resp. $|C_{10}^c|$. Thus, Proposition \ref{prop-la} completes the proof.
\proofend

It is natural to ask

\smallskip

\smallskip{\bf Question:} {\it What is the minimal degree of conic-line arrangements in anti-Zariski pairs?}
\smallskip

\medskip

A related question is

\smallskip

\smallskip{\bf Question:} {\it Are there any degree $6$ conic-line arrangements that form a reducible
equisingular family?}

\medskip
Recall that all equisingular families of curves of degree $\le5$ are irreducible \cite{Deg90}, and there are reducible families of conic-line arrangements of degree $7$ \cite{Tok} (see also \cite{AS}).

\subsection{Construction of new anti-Zariski pairs of reducible curves}\label{sec-red}

\begin{proposition}\label{prop-red}
Let $C,C^c$ be an anti-Zariski pair of complex conjugate curves, $D$ an irreducible curve that cannot be joined by an equisingular deformation with any of the components of $C^c$. Then $C\cup D$ and $(C\cup D)^c$ form an anti-Zariski pair.
\end{proposition}

{\bf Proof.}
Indeed, under the hypotheses of the proposition, an equisingular deformation of $C\cup D$ into $(C\cup D)^c$ would descend to an equisingular deformation of $C$ into $C^c$.
\proofend

\begin{remark}
The hypotheses of Proposition \ref{prop-red} are fulfilled, for example, if $\deg D$ differs from the degree of any components of $C$, or if $D$ contains a singular point topologically different from any singular point of any component of $C$.
\end{remark}

\subsection{Anti-Zariski pairs of curves over a non-Archimedean field}\label{sec-na}

Denote by $\K$ the field of locally convergent or formal complex Puiseux series, or the fields of power series
$$a(t)=\sum_{r\in A}a_rt^r,\quad a_r\in\C,\ r\in A,$$
where $A$ ranges over all well-ordered subsets of $\R$. In each case, $\K$ is algebraically closed and possesses a real non-Archimedean valuation
$$\val\left(\sum_ra_rt^r\right)=-\min\{r\ : \ a_r\ne0\},$$
and the norm
$$\|a(t)\|=\exp(\val(a(t))).$$
It is well-known that over $\C$, the topological type of an isolated plane curve singular point $(C,p)$ is completely characterized by the following discrete invariant: the list of irreducible components of $(C,p)$, the list of pairwise intersection multiplicities of these components, and the list of characteristic Puiseux exponents of the components. Thus, we use this invariant to define the topological equivalence of singular points and the equisingularity relation over $\K$.

\begin{proposition}\label{prop-na}
Let curves $C_1,C_2\subset\PP^2$ belong to different components of the same equisingular family, and let $C_1,C_2$ be defined over isomorphic number fields so that they are interchanged by the field isomorphism. Denote by $\widehat C_1,\widehat C_2$ the curves in $\PP^2_\K$ defined by the same equations. Then $\widehat C_1,\widehat C_2$ form an anti-Zariski pair over $\K$.
\end{proposition}

{\bf Proof.}
Equisingular families are defined over the field of algebraic numbers $\overline\Q$; hence, by Lefschetz principle, $\widehat C_1$ and $\widehat C_2$ belong to different components of the equisingular family over $\K$. Furthermore, by Steinitz theorem, an isomorphism of number fields extends to an automorphism of $\C$, which induces an isometry $\PP^2_\K\to\PP^2_\K$ taking $\widehat C_1$ to $\widehat C_2$.
\proofend

As examples, we present here anti-Zariski tuples of equisingular maximizing sextics over $\K$. In each tuple, the curves are defined over isomorphic number fields, and the curves are interchanged by these field isomorphisms. We cite the results from \cite{Ore}: for each collection of singularities listed below, there is an irreducible 
polynomial $P(x)\in\Q[x]$ of some degree $n$ and an $n$-tuple of maximizing sextic curves defined over the number fields $\Q(\xi)$ with $\xi$ ranging over the roots of $P(x)$ so that the curves are interchanged by the corresponding field isomorphisms:
\begin{itemize}\item an anti-Zariski triple over $\K$
$$A_{18}\oplus A_1,\quad P(x)=x^3-2x-2;$$
\item an anti-Zariski triple over $\K$
$$A_{16}\oplus A_2\oplus A_1,\quad P(x)=x^3 - x^2 + x - 3;$$
\item an anti-Zariski $6$-tuple over $\K$
$$A_{14}\oplus A_4\oplus A_1,\quad P(x)=2x^6-6x^5+10x^4-5x^3-10x^2+4x+8;$$
\item an anti-Zariski $4$-tuple over $\K$
$$A_{13}\oplus A_6,\quad P(x)=x^4+7;$$
\item an anti-Zariski triple over $\K$
$$A_{12}\oplus A_6\oplus A_1,\quad P(x)=7x^3-16x^2+12x-4;$$
\item an anti-Zariski triple over $\K$
$$A_{12}\oplus A_4\oplus A_2\oplus A_1,\quad P(x)=15x^3-48x^2+40x-10;$$
\item an anti-Zariski triple over $\K$
$$A_{10}\oplus A_8\oplus A_1,\quad P(x)=x^3-4x^2+8x-4;$$
\item an anti-Zariski triple over $\K$
$$A_{10}\oplus A_6\oplus A_2\oplus A_1,\quad P(x)=x^3-x^2+3;$$
\item an anti-Zariski triple over $\K$
$$A_{10}\oplus 2A_4\oplus A_1,\quad P(x)=x^3-x^2-x-1;$$
\item an anti-Zariski triple over $\K$
$$A_9\oplus A_6\oplus A_4,\quad P(x)=x^3-5x-5;$$
\item an anti-Zariski triple over $\K$
$$A_8\oplus A_6\oplus A_4\oplus A_1,\quad P(x)=x^3-x^2-x+5.$$
\end{itemize}

\end{document}